\documentclass[leqno,12pt]{article} 
\usepackage{amsmath, amssymb}
\usepackage{amsthm} 
\newcommand\supp{\mathop{\rm supp}}
\newcommand\esssup{\mathop{\rm ess \, sup}}
\newcommand\essinf{\mathop{\rm ess \, inf}}

\theoremstyle{plain} 
\newtheorem{theorem}{\indent\sc Theorem}[section]
\newtheorem{lemma}[theorem]{\indent\sc Lemma}

\newtheorem{proposition}[theorem]{\indent\sc Proposition}

\theoremstyle{definition} 
\newtheorem{definition}[theorem]{\indent\sc Definition}
\newtheorem{remark}[theorem]{\indent\sc Remark}
\newtheorem{example}[theorem]{\indent\sc Example}

\title{Fractional type operators on the Heisenberg group} 
\author{
\text{Pablo Rocha} 
}
\date{e-mail: pablo.rocha@uns.edu.ar} 
\begin{document}

\maketitle


\begin{abstract}
Let $\rho(\cdot)$ be the Koranyi norm on the Heisenberg group $\mathbb{H}^{n} \equiv (\mathbb{R}^{2n} \times \mathbb{R}, \, \cdot \, )$ defined by
\[
\rho(x,t) = \left( |x|^{4} + 16 t^{2} \right)^{1/4}, \,\,\,\, (x,t) \in \mathbb{H}^{n}.
\] 
For $0 \leq \alpha < Q:=2n+2$, $m \in \mathbb{N} \cap \left(1 - \frac{\alpha}{Q}, \infty \right)$, and $m$ positive constants 
$\alpha_1, ..., \alpha_m$ such that $\alpha_1 + \cdot \cdot \cdot + \alpha_m = Q - \alpha$, we consider the following generalization of the Riesz potential on $\mathbb{H}^{n}$ 
\[
T_{\alpha, \, m}f(x,t) = \int_{\mathbb{H}^{n}} f(y,s) \prod_{j=1}^{m} \rho\left((A_j y, r_j^{-2} s)^{-1} \cdot ( x, t)\right)^{-\alpha_j} \, dy \, ds, 
\]
where, in the case $0 < \alpha < Q$, the $A_j$'s are matrices belonging to $Sp (2n, \mathbb{R}) \cap SO(2n)$ and $r_j = 1$ for every 
$j=1, ..., m$; for $\alpha = 0$, we consider $A_j = r_j^{-1} \, I_{2n \times 2n}$ for every $j=1, ..., m$, where the $r_j$'s are positive
constants such that $r_{i}^{2} - r_{j}^{2} \neq 0$ if $i \neq j$. In this note we study the behavior of these operators on variable 
Hardy spaces in $\mathbb{H}^{n}$.
\end{abstract}

\section{Introduction}
Let $J$ be the $2n \times 2n$ skew-symmetric matrix given by 
\begin{equation} \label{matrix J}
J = \frac{1}{2} \left( \begin{array}{cc}
                           0 & -I_n \\
                           I_n & 0 \\
                           \end{array} \right),
\end{equation}
where $I_{n}$ is the $n \times n$ identity matrix. The Heisenberg group $\mathbb{H}^{n}$ is $\mathbb{R}^{2n} \times \mathbb{R}$ endowed with the following group law (noncommutative)
\begin{equation} \label{group law}
(x,t) \cdot (y,s) = (x+y, \, t+s + x^{t}Jy).
\end{equation}
With this structure we have that $e=(0,0)$ is the neutral element and $(x,t)^{-1} = (-x, -t)$ is the inverse of $(x,t)$. 
																			
The \textit{Koranyi norm} on the Heisenberg group $\mathbb{H}^{n}$ is defined by
\[
\rho(x,t) = \left( |x|^{4} + 16 \, t^{2}  \right)^{1/4}, \,\,\, (x,t) \in \mathbb{H}^{n}.
\]
For $0 < \alpha < Q := 2n+2$ and $(x,t) \in \mathbb{H}^{n}$, the Riesz potential $\mathcal{R}_{\alpha}$ on $\mathbb{H}^{n}$ is defined 
by right-convolution with the kernel $\rho(\cdot)^{\alpha - Q}$ 
\[
\mathcal{R}_{\alpha}f (x,t) = \left( f \ast \rho(\cdot)^{\alpha - Q} \right)(x,t) = \int_{\mathbb{H}^{n}} f(y,s) 
\rho \left((y,s)^{-1} \cdot (x,t) \right)^{\alpha - Q} dy \, ds.
\]
It is well known that the Heisenberg group is a homogeneous group in the sense of G. Folland and E. Stein \cite{Folland}. There the 
authors defined Hardy spaces on homogeneous groups, $H^{p_0}(G)$ ($0 < p_0 < \infty$), and studied the behavior of operators defined by 
right-convolution with kernels of type $(\alpha, N)$ on the spaces $H^{p}(G)$ (see \cite[p. 184]{Folland}). In particular, when 
$G = \mathbb{H}^{n}$, we have that the kernel $\rho(\cdot)^{\alpha - Q}$ is of type $(\alpha, N)$ for $0 < \alpha < Q$ and every 
$N \in \mathbb{N}$. Thus, from \cite[Proposition 6.2]{Folland}, it follows that the Reisz potential $\mathcal{R}_{\alpha}$ is bounded from 
$L^{p_0}(\mathbb{H}^{n})$ into $L^{q_0}(\mathbb{H}^{n})$ for $1 < p_0 < Q/\alpha$ and $1/q_0 = 1/p_0 - \alpha/Q$ (see also 
\cite[p. 198]{Krantz}); and, by \cite[Theorem 6.10]{Folland}, we have that $\mathcal{R}_{\alpha}$ is bounded from $H^{p_0}(\mathbb{H}^{n})$ into $H^{q_0}(\mathbb{H}^{n})$ for $0 < p_0 \leq 1$ and $1/q_0 = 1/p_0 - \alpha/Q$. 

The variable Hardy spaces on the Heisenberg group $H^{p(\cdot)}(\mathbb{H}^{n})$ were defined by J. Fang and J. Zhao in \cite{Fang}. One of their main goals is the atomic characterization of $H^{p(\cdot)}(\mathbb{H}^{n})$, as an application of such characterization they proved the boundedness on $H^{p(\cdot)}(\mathbb{H}^{n})$ of singular integrals. Recently, with this framework, the author in \cite{Pablo}, for 
$0 \leq \alpha < Q=2n+2$, $N \in \mathbb{N}$ and H\"older exponent functions $p(\cdot) : \mathbb{H}^{n} \to (0, +\infty)$ such that 
$\frac{Q}{Q+N} < p_{-} \leq p(\cdot) \leq p_{+} < \frac{Q}{\alpha}$, proved the $H^{p(\cdot)}(\mathbb{H}^{n}) \to L^{q(\cdot)}(\mathbb{H}^{n})$ and $H^{p(\cdot)}(\mathbb{H}^{n}) \to H^{q(\cdot)}(\mathbb{H}^{n})$ boundedness of right-convolution operators with kernels of type 
$(\alpha, N)$ on $\mathbb{H}^{n}$, where $\frac{1}{q(\cdot)} = \frac{1}{p(\cdot)} - \frac{\alpha}{Q}$. In particular, the Riesz potential 
$\mathcal{R}_{\alpha}$ ($0 < \alpha < Q$) on $\mathbb{H}^{n}$ satisfies such estimates.

In this note, for $0 \leq \alpha < Q=2n+2$, $m \in \mathbb{N} \cap \left(1 - \frac{\alpha}{Q}, \infty \right)$, and $m$ positive constants 
$\alpha_1, ..., \alpha_m$ such that $\alpha_1 + \cdot \cdot \cdot + \alpha_m = Q - \alpha$, we consider the following generalization of the Riesz potential on $\mathbb{H}^{n}$ 
\begin{equation} \label{Tm}
T_{\alpha, \, m}f(x,t) = \int_{\mathbb{H}^{n}} f(y,s) \prod_{j=1}^{m} \rho\left((A_j y, r_j^{-2} s)^{-1} \cdot (x, t)\right)^{-\alpha_j} \, dy \, ds, 
\end{equation}
where, in the case $0 < \alpha < Q$, the $A_j$'s are matrices belonging to $Sp (2n, \mathbb{R}) \cap SO(2n)$ and $r_j = 1$ for every $j=1, ..., m$; for $\alpha = 0$, we consider $A_j = r_j^{-1} \, I_{2n \times 2n}$ for every $j=1, ..., m$, where the $r_j$'s are positive constants such that $r_{i}^{2} - r_{j}^{2} \neq 0$ if $i \neq j$. We observe that if $0 < \alpha < Q$, $m=1$ and 
$A_1 = I_{2n \times 2n}$, then $T_{\alpha, \, 1} = \mathcal{R}_{\alpha}$; and if $0 \leq \alpha < Q$ and $m \geq 2$, then the 
operator $T_{\alpha, m}$ is not a convolution type operator. The operator given in (\ref{Tm}) is inspired by the following operator 
defined on $\mathbb{R}^{n}$
\begin{equation} \label{Um}
U_{\alpha, \, m} f(x) = \int_{\mathbb{R}^{n}} f(y) \prod_{j=1}^{m} |x - A_j y|^{-\alpha_j} \, dy,
\end{equation}
where $0 \leq \alpha < n$, $m \in \mathbb{N} \cap (1- \frac{\alpha}{n}, \infty)$, $| \cdot |$ is the Euclidean norm on $\mathbb{R}^{n}$, the matrices $A_j$'s are invertible with $A_i - A_j$ also invertible if $i \neq j$, and the $\alpha_j$'s are positive numbers such that 
$\alpha_1 + \cdot \cdot \cdot + \alpha_m = n - \alpha$. The operator $U_{\alpha, \, m}$ given by (\ref{Um}) is a kind of generalization of the Riesz potential on $\mathbb{R}^{n}$. In \cite{RoUr}, the author jointly with M. Urciuolo proved the 
$H^{p_0}(\mathbb{R}^{n}) \to L^{q_0}(\mathbb{R}^{n})$ boundedness of $U_{\alpha, m}$ for $0 < p_0 < n/\alpha$ and $1/q_0 = 1/p_0 - \alpha/n$; we also showed that the $H^{p_0}(\mathbb{R}) \to H^{q_0}(\mathbb{R})$ boundedness does not hold for $0 < p_0 \leq \frac{1}{1+\alpha}$, 
$\frac{1}{q_0} = \frac{1}{p_0} - \alpha$ and $U_{\alpha, m}$ defined on $\mathbb{R}$ with $0< \alpha < 1$, $m=2$, $A_1=1$ and $A_2 = -1$. This 
is a significant difference with respect to the case $0< \alpha < 1$, $n=m=1$ and $A_1=1$; indeed, in this case the operator $U_{\alpha, m}$ coincides with the Riesz potential on $\mathbb{R}$, which is bounded from $H^{p_0}(\mathbb{R})$ into $H^{q_0}(\mathbb{R})$ as it is well known (see \cite{Taibleson, Skrantz}). In the context of variable exponents (see \cite{Nakai, Cruz-Uribe2}), the author and M. Urciuolo 
in \cite{RochaUr} proved the $H^{p(\cdot)}(\mathbb{R}^{n}) \to L^{q(\cdot)}(\mathbb{R}^{n})$ boundedness of the operator $U_{\alpha, m}$ and the $H^{p(\cdot)}(\mathbb{R}^{n}) \to H^{q(\cdot)}(\mathbb{R}^{n})$ boundedness of Riesz potential via the infinite atomic and molecular decomposition developed in \cite{Nakai}. In \cite{Rocha3}, the author gave another proof of the results obtained in \cite{RochaUr}, but by using the finite atomic decomposition developed in \cite{Cruz-Uribe2}.

In this note, by using the frameworks developed in \cite{Folland} and \cite{Fang} joint with some additional results obtained by the author in \cite{Pablo}, we will prove the following theorem.

\vspace{0.3cm}
{\sc Theorem} \ref{Hp-Lq} 
{\it Let $0 \leq \alpha < Q$, $m \in \mathbb{N} \cap (1- \frac{\alpha}{Q}, \infty)$ and let $T_{\alpha, \, m}$ be the operator 
defined by (\ref{Tm}). Suppose $p(\cdot) \in \mathcal{P}^{\log}(\mathbb{H}^{n})$ 
such that $p(A_j x, r_j^{-2} t) = p(x,t)$ for all $(x,t)$ and all $j=1, ..., m$, and $0 < p_{-} \leq p_{+} < \frac{Q}{\alpha}$. 
If $\frac{1}{q(\cdot)} = \frac{1}{p(\cdot)} - \frac{\alpha}{Q}$, then $T_{\alpha, \, m}$ can be extended to a bounded 
operator from $H^{p(\cdot)}(\mathbb{H}^{n})$ into $L^{q(\cdot)}(\mathbb{H}^{n})$.}

\vspace{0.3cm}
This paper is organized as follows. Section 2 introduces the Heisenberg group $\mathbb{H}^{n}$ and also gives some supporting results about the operator $T_{\alpha, m}$, which are crucial to prove the Theorem \ref{Hp-Lq}. In Section 3, we establish the basics on variable Lebesgue and Hardy spaces. Finally, in Section 4 we prove the Theorem \ref{Hp-Lq} and also give, for the case $0 < \alpha < Q$, non-trivial examples of variable exponents $p(\cdot)$ satisfying the condition $p(A_j x, t) = p(x,t)$ for all $(x,t)$ and all
$j=1, ..., m$.

\vspace{0.3cm}
\textbf{Notation:} The symbol $A\lesssim B$ stands for the inequality $A \leq cB$ for some positive constant $c$. The symbol $A \approx B$ stands for $B \lesssim A \lesssim B$. For a measurable subset $E\subseteq \mathbb{H}^{n}$ we denote by $\left\vert E\right\vert $ and 
$\chi_{E}$ the Haar measure of $E$ and the characteristic function of $E$ respectively. Throughout this paper, $C$ will denote a positive real constant not necessarily the same at each occurrence.

\section{Preliminaries}

We recall that the Heisenberg group is $\mathbb{H}^{n} = (\mathbb{R}^{2n} \times \mathbb{R}, \, \cdot \,)$ with the group law $\cdot$ given by 
(\ref{group law}). The dilations in $\mathbb{H}^{n}$ are given by
\[
r \cdot (x,t) = (rx, r^{2}t), \,\,\,\, r>0,
\]
and satisfy
\[
r \cdot ( (x,t) \cdot (y,s) ) = (r \cdot (x,t)) \cdot (r \cdot (y,s)).
\]

$M_{n}(\mathbb{R})$ denotes the algebra of all matrices of degree $n$ with coefficients in $\mathbb{R}$. The symplectic group 
$Sp(2n, \mathbb{R})$ and the special orthogonal group $SO(2n)$ are defined by
\[
Sp(2n, \mathbb{R}) = \{ A \in M_{2n}(\mathbb{R}) : A^{t} J A = J\},
\]
where $J$ is the $2n \times 2n$ skew-symmetric matrix given in (\ref{matrix J}), and
\[
SO(2n) = \{ A \in M_{2n}(\mathbb{R}) : A^{t} = A^{-1} \,\, \text{and} \,\, \det(A) = 1 \}.
\]
The group $Sp(2n, \mathbb{R}) \cap SO(2n)$ acts on $\mathbb{H}^{n}$ by
\[
A \cdot (x,t) = (Ax, t), \,\,\,\, A \in Sp(2n, \mathbb{R}) \cap SO(2n) \,\, \text{and} \,\, (x,t) \in \mathbb{H}^{n};
\]
and satisfy
\[
A \cdot ( (x,t) \cdot (y,s) ) = (A \cdot (x,t)) \cdot (A \cdot (y,s)).
\]

The topology in $\mathbb{H}^{n}$ is the induced by $\mathbb{R}^{2n} \times \mathbb{R}$, so the borelian sets of $\mathbb{H}^{n}$ are identified with those of $\mathbb{R}^{2n} \times \mathbb{R}$. The Haar measure in $\mathbb{H}^{n}$ is the Lebesgue measure of $\mathbb{R}^{2n} \times \mathbb{R}$, thus $L^{p_0}(\mathbb{H}^{n}) \equiv L^{p_0}(\mathbb{R}^{2n} \times \mathbb{R})$, 
$0 < p_0 \leq \infty$. Moreover, for $f \in L^{1}(\mathbb{H}^{n})$ and each $(y,s) \in \mathbb{H}^{n}$
\begin{equation} \label{invariant transl}
\int_{\mathbb{H}^{n}} f((y,s) \cdot (x,t)) \, dx \, dt = \int_{\mathbb{H}^{n}} f((x,t) \cdot (y,s)) \, dx \, dt = 
\int_{\mathbb{H}^{n}} f(x,t) \, dx \, dt, 
\end{equation}
for $r > 0$ fixed, we also have
\[
\int_{\mathbb{H}^{n}} f(r \cdot (x,t)) \, dx \, dt = r^{-Q} \int_{\mathbb{H}^{n}} f(x,t) \, dx \, dt,
\]
where $Q= 2n+2$. The number $2n+2$ is known as the {\it homogeneous dimension} of $\mathbb{H}^{n}$ (we observe that the {\it topological dimension} of $\mathbb{H}^{n}$ is $2n+1$).

The {\it Koranyi norm} on $\mathbb{H}^{n}$ is the function $\rho : \mathbb{H}^{n} \to [0, \infty)$ defined by
\begin{equation} \label{Koranyi norm}
\rho(x,t) = \left( |x|^{4} + 16 \, t^{2}  \right)^{1/4}, \,\,\, (x,t) \in \mathbb{H}^{n},
\end{equation}
where $| \cdot |$ is the usual Euclidean norm on $\mathbb{R}^{2n}$. The Koranyi norm satisfies the following properties
\begin{eqnarray*}
\rho(x,t) & = & 0 \,\,\, \text{if and only if} \,\, (x,t)=(0,0), \\
\rho((x,t)^{-1}) & = & \rho(x,t) \,\,\,\, \text{for all} \,\, (x,t) \in \mathbb{H}^{n} \\
\rho(r \cdot (x,t)) & = & r \rho(x,t) \,\,\,\, \text{for all} \,\, (x,t) \in \mathbb{H}^{n} \,\, \text{and all} \,\, r > 0, \\
\rho(A \cdot (x,t)) & = & \rho(x,t) \,\,\,\, \text{for all} \,\, A \in Sp(2n, \mathbb{R}) \cap SO(2n)  \,\, \text{and all} \,\, (x,t) \in \mathbb{H}^{n}, \\
\rho((x,t) \cdot (y,s)) & \leq & \rho(x,t) + \rho(y,s) \,\,\,\, \text{for all} \,\, (x,t), \, (y,s) \in \mathbb{H}^{n}, \\
| \rho(x,t) - \rho(y,s) | & \leq & \rho((x,t) \cdot (y,s)) \,\,\,\, \text{for all} \,\, (x,t), \, (y,s) \in \mathbb{H}^{n}.
\end{eqnarray*}
Moreover, $\rho$ is continuous on $\mathbb{H}^{n}$ and is smooth on $\mathbb{H}^{n} \setminus \{ e \}$. The $\rho$ - ball centered at 
$(x_0, t_0) \in \mathbb{H}^{n}$ with radius $\delta > 0$ is defined by
\[
B_{\delta}(x_0,t_0) := \{ (y,s) \in \mathbb{H}^{n} : \rho((x_0, t_0)^{-1} \cdot (y,s)) < \delta \}.
\]

\begin{remark}
The topology in $\mathbb{H}^{n}$ induced by the $\rho$ - balls coincides with the Euclidean topology of $\mathbb{R}^{2n} \times \mathbb{R}$ (see
\cite[Proposition 3.1.37]{Fischer}). In particular, the $\rho$ - balls are measurable sets of $\mathbb{H}^{n}$.
\end{remark}

Let $|B_{\delta}(x_0, t_0)|$ be the Haar measure of the $\rho$ - ball $B_{\delta}(x_0, t_0) \subset \mathbb{H}^{n}$. Then, 
\[
|B_{\delta}(x_0, t_0)| = c_0 \delta^{Q},
\]
where $c_0 = |B_1(0,0)|$ and  $Q = 2n+2$. Moreover, for any $(x,t), (x_0, t_0) \in \mathbb{H}^{n}$ and $\delta >0$, we have
$(x, t) \cdot B_{\delta}(x_0, t_0) = B_{\delta}((x,t) \cdot (x_0, t_0))$.

\begin{remark} \label{cambio de centro}
If $f \in L^{1}(\mathbb{H}^{n})$, then for every $\rho$ - ball $B$ and every $(x_0, t_0) \in \mathbb{H}^{n}$, by (\ref{invariant transl}), we have
\[
\int_{B} f(x,t) \, dxdt = \int_{(x_0, t_0)^{-1} \cdot B} f((x_0, t_0) \cdot (x,t)) \, dxdt.
\]
\end{remark}

For $0 < \alpha < Q$ and a locally integrable function $f$ on $\mathbb{H}^{n}$, we define the fractional maximal function $M_{\alpha} f$ 
on $\mathbb{H}^{n}$ by
\[
M_{\alpha}f(x,t) = \sup_{B \ni (x,t)} |B|^{\frac{\alpha}{Q} - 1}\int_{B} |f(y,s)| \, dy \, ds,
\]
where the supremum is taken over all the $\rho$ - balls $B$ containing $(x,t)$. For $\alpha=0$, we have that $M_0 = M$, where $M$ is the Hardy-Littlewood maximal operator on $\mathbb{H}^{n}$. 

\begin{proposition} \label{fract max op} If $0 \leq \alpha < Q$, $1 < p_0 < Q/\alpha$ and $1/q_0 = 1/p_0 - \alpha/Q$, then
\[
\left( \int_{\mathbb{H}^{n}} [M_{\alpha}f(x,t)]^{q_0} dx \, dt \right)^{1/q_0} \leq C \left( \int_{\mathbb{H}^{n}} |f(x,t)|^{p_0} dx \, dt \right)^{1/p_0}.
\]
Moreover, for the case $\alpha = 0$, one also has that $\|M_{0} f\|_{L^{\infty}(\mathbb{H}^{n})} \leq \| f\|_{L^{\infty}(\mathbb{H}^{n})}$ and
$M_0$ is weak type $(1,1)$.
\end{proposition}

\begin{proof}
The case $\alpha = 0$ follows from \cite[Note 2.5 in p. 11]{Stein} and \cite[Theorem 1 in p. 13]{Stein}. Now, the case $0 < \alpha < Q$ follows from
\cite[Proposition 4.3]{Pablo} applied with $\omega \equiv 1$.
\end{proof}

The following result talks about the $L^{p_0}(\mathbb{H}^{n}) \to L^{q_0}(\mathbb{H}^{n})$ boundedness of the operator $T_{\alpha, \, m}$.

\begin{theorem} \label{Hp0-Lq0}
If $0 \leq \alpha < Q$, $m \in \mathbb{N} \cap (1 - \frac{\alpha}{Q}, \infty)$, and $T_{\alpha, \, m}$ is the operator defined by (\ref{Tm}), then 
$T_{\alpha, \, m}$ is bounded from $L^{p_0}(\mathbb{H}^{n})$ into $L^{q_0}(\mathbb{H}^{n})$ for $1 < p_0 < Q/\alpha$ and 
$1/q_0 = 1/p_0 - \alpha/Q$.
\end{theorem}

\begin{proof}
We study the cases $0 < \alpha < Q$ and $\alpha = 0$ separately. For $0 < \alpha < Q$, we have that $A_j \in Sp(2n, \mathbb{R}) \cap SO(2n)$ for every $j=1, ..., m$, and so
\[
|T_{\alpha, \, m}f(x,t)| \leq \sum_{j=1}^{m} \int_{\mathbb{H}^{n}} |f(y,s)| \, \rho\left((A_j y,s)^{-1} \cdot (x, t) \right)^{\alpha - Q} dy \, ds = \sum_{j=1}^{m} \mathcal{R}_{\alpha}|f|(A_j^{-1} x, t),
\]
Then, the $L^{p_0}(\mathbb{H}^{n}) \to L^{q_0}(\mathbb{H}^{n})$ boundedness for $T_{\alpha, \, m}$ in the case $0 < \alpha < Q$ follows from the boundedness of the Riesz potential $\mathcal{R}_{\alpha}$.

Now we study the case $\alpha = 0$. Let $\beta = \min \{ |r_i - r_j|, \, |r_i^{2} - r_j^{2}|^{1/2} : i \neq j \}$ and 
$\gamma = \max \{ |r_i| : i =1, 2, ..., m \}$. For $(x,t) \neq (0,0)$ and $1 \leq j \leq m$ we define the following sets
\[
\Omega_j = \left\{ (y,s) \in \mathbb{H}^{n} : \rho \left( (y,s)^{-1} \cdot (r_j x, r_j^{2} t) \right) < \frac{\beta}{2} \rho(x,t) \right\};
\]
we also define
\[
\Omega_{m+1} = \left( \bigcap_{j=1}^{m} \Omega_{j}^{c} \right) \cap \left\{ (y,s) \in \mathbb{H}^{n} : \rho(y,s) \leq (1 + \gamma)
\rho(x,t) \right\}
\]
and
\[
\Omega_{m+2} = \left( \bigcap_{j=1}^{m} \Omega_{j}^{c} \right) \cap \left\{ (y,s) \in \mathbb{H}^{n} : \rho(y,s) > (1 + \gamma)
\rho(x,t) \right\}
\]
Then, for $(x,t) \neq (0,0)$ and $C = r_1^{-\alpha_1} \cdot \cdot \cdot r_{m}^{-\alpha_m}$
\[
T_{0, \, m} f(x,t) = C \sum_{j=1}^{m+2} \int_{\Omega_j}  f(y,s) \prod_{k=1}^{m} \rho\left((y,s)^{-1} \cdot (r_k x, r_{k}^{2}t)\right)^{-\alpha_k} \, dy \, ds.
\]
Next, we fix $(x,t) \neq (0,0)$ and $1 \leq j \leq m$. For $(y,s) \in \Omega_j$ and $i \neq j$ we have that $\rho \left( (y,s)^{-1} \cdot 
(r_j x, r_j^{2} t) \right) < \frac{\beta}{2} \rho(x,t)$ and $\rho \left( (r_i x, r_i^{2} t)^{-1} \cdot (r_j x, r_j^{2} t) \right) \geq \beta \rho(x,t)$; to obtain this last inequality we use that $x J x^{t} = 0$. So,
\[
\rho \left( (y,s)^{-1} \cdot (r_i x, r_i^{2} t) \right) \geq \rho \left( (r_i x, r_i^{2} t)^{-1} \cdot (r_j x, r_j^{2} t) \right) -
\rho \left( (y,s)^{-1} \cdot (r_j x, r_j^{2} t) \right) \geq \frac{\beta}{2} \rho(x,t)
\]
and
\begin{eqnarray*}
&& C^{-1}|T_{0,m}(\chi_{\Omega_j}f)(x,t)| \leq \int_{\Omega_j} |f(y,s)| \prod_{k=1}^{m} \rho\left((y,s)^{-1} \cdot (r_k x, r_{k}^{2}t)\right)^{-\alpha_k} \, dy \, ds \\
&\leq& \left(\frac{\beta}{2} \rho(x,t) \right)^{\alpha_j - Q} \int_{\Omega_j} \rho\left((y,s)^{-1} \cdot (r_j x, r_{j}^{2}t)\right)^{-\alpha_j} |f(y,s)| \, dy \, ds \\
&\leq& \left(\frac{\beta}{2} \rho(x,t) \right)^{\alpha_j - Q} \sum_{k=0}^{\infty} \int_{\frac{\beta \rho(x,t)}{2^{k+2}} \leq 
\rho \left((y,s)^{-1} \cdot (r_j x, r_{j}^{2}t)\right) < \frac{\beta \rho(x,t)}{2^{k+1}}} \rho\left((y,s)^{-1} \cdot (r_j x, r_{j}^{2}t)\right)^{-\alpha_j}
\end{eqnarray*}
\[
\times \,\, |f(y,s)| \, dy \, ds
\]
\begin{eqnarray*}
&\leq& \left(\frac{\beta}{2} \rho(x,t) \right)^{\alpha_j - Q} \sum_{k=0}^{\infty} \int_{\frac{\beta\rho(x,t)}{2^{k+2}} \leq 
\rho \left((y,s)^{-1} \cdot (r_j x, r_{j}^{2}t)\right) < \frac{\beta \rho(x,t)}{2^{k+1}}} \left( \frac{\beta \rho(x,t)}{2^{k+2}} 
\right)^{-\alpha_j} |f(y,s)| \, dy \, ds 
\end{eqnarray*}
\begin{eqnarray*}
&\leq& 2^{Q} c_0 \sum_{k=0}^{\infty} 2^{-(k+1)(Q-\alpha_j)} c_0^{-1}\left( \frac{\beta \rho(x,t)}{2^{k+1}} \right)^{-Q}
\int_{\rho \left((y,s)^{-1} \cdot (r_j x, r_{j}^{2}t)\right) < \frac{\beta \rho(x,t)}{2^{k+1}}}  |f(y,s)| \, dy \, ds 
\end{eqnarray*}
\begin{eqnarray*}
&\leq& 2^{\alpha_j} c_0 \sum_{k=0}^{\infty} 2^{-k(Q-\alpha_j)} c_0^{-1}\left( \frac{\beta \rho(x,t)}{2^{k+1}} \right)^{-Q}
\int_{\rho \left((y,s)^{-1} \cdot (r_j x, r_{j}^{2}t)\right) < \frac{\beta \rho(x,t)}{2^{k+1}}}  |f(y,s)| \, dy \, ds, 
\end{eqnarray*}
where $c_0 = |B_{1}(0,0)|$. Then
\begin{equation} \label{omegaj}
C^{-1}|T_{0,m}(\chi_{\Omega_j}f)(x,t)| \leq \frac{2^{\alpha_j} c_0}{1- 2^{-(Q-\alpha_j)}} M_0 f(r_j x, r_j^{2} t), \,\,\,\, \text{for every} \,\, 
j=1, ...,m.
\end{equation}

On the region $\Omega_{m+1}$ we have
\begin{eqnarray*}
C^{-1}|T_{0, m}(\chi_{\Omega_{m+1}}f)(x,t)| &\leq& \int_{\Omega_{m+1}} |f(y,s)| \prod_{k=1}^{m} \rho\left((y,s)^{-1} \cdot (r_k x, r_{k}^{2}t)\right)^{-\alpha_k} \, dy \, ds \\
&\leq& \left( \frac{2}{\beta} \right)^{Q} \rho(x,t)^{-Q} \int_{\rho(y,s) \leq (1+\gamma) \rho(x,t)} |f(y,s)| \, dy \, ds \\
&\leq& \left( \frac{2}{\beta} \right)^{Q} \rho(x,t)^{-Q} \int_{\rho((x,t)^{-1} \cdot (y,s)) < 2(1+\gamma) \rho(x,t)} |f(y,s)| \, dy \, ds,
\end{eqnarray*}
so
\begin{equation} \label{omegam1}
C^{-1}|T_{0, m}(\chi_{\Omega_{m+1}}f)(x,t)| \leq \left( \frac{4 (1+\gamma)}{\beta} \right)^{Q} c \, M_0 f(x,t).
\end{equation}

Now, if $(y,s) \in \Omega_{m+2}$, for every $1 \leq i \leq m$, we have that $\rho((y,s)^{-1} \cdot (r_i x, r_i^{2} t)) \geq 
\frac{1}{1+\gamma} \rho(y,s)$. Thus
\[
C^{-1}|T_{0,m} (\chi_{\Omega_{m+2}} f)(x,t)| \leq \int_{\Omega_{m+2}} |f(y,s)| \prod_{k=1}^{m} \rho\left((y,s)^{-1} \cdot (r_k x, r_{k}^{2}t)\right)^{-\alpha_k} \, dy \, ds
\]
\[
\leq (1+\gamma)^{Q} \int_{\rho(y,s) > (1+\gamma) \rho(x,t)} |f(y,s)| \rho(y,s)^{-Q} \, dy \, ds
\]
\[
\leq (1+\gamma)^{Q/p_0'} \, \frac{\sigma(\{ \rho = 1\})^{1/p_0'}}{((p_0'-1)Q)^{1/p_0'}} \, \| f \|_{L^{p_0}(\mathbb{H}^{n})} \, 
\rho(x,t)^{-Q/p_0},
\]
to obtain this last inequality we apply H\"older's inequality and polar coordinates on $\mathbb{H}^{n}$ (see \cite[p. 211]{Krantz}). Then, for
$\lambda > 0$
\begin{equation} \label{omegam2}
\left| \left\{ (x,t) \in \mathbb{H}^{n} : C^{-1}|T_{0,m} (\chi_{\Omega_{m+2}} f)(x,t)| > \lambda) \right\} \right| \lesssim
\left( \frac{\| f \|_{L^{p_0}(\mathbb{H}^{n})}}{\lambda}  \right)^{p_0}.
\end{equation}
Thus, by (\ref{omegaj}), (\ref{omegam1}), Proposition \ref{fract max op}, and (\ref{omegam2}) we have that $T_{0,m}$ is weak type $(p_0, p_0)$ for every $1 \leq p_0 < \infty$. Finally, the Marcinkiewicz interpolation theorem (see \cite[p. 34]{Grafakos}) allows to conclude the theorem.
\end{proof}

For every $i = 1,2, ..., 2n+1$, $X_i$ denotes the left-invariant vector field on $\mathbb{H}^{n}$, which are given by
\[
X_i = \frac{\partial}{\partial x_i} + \frac{x_{i+n}}{2} \frac{\partial}{\partial t}, \,\,\,\, i=1, 2, ..., n;
\]
\[
X_{i+n} = \frac{\partial}{\partial x_{i+n}} - \frac{x_{i}}{2} \frac{\partial}{\partial t}, \,\,\, i=1, 2, ..., n;
\]
and
\[
X_{2n+1} = \frac{\partial}{\partial t}.
\]

Given a multiindex $I=(i_1,i_2, ..., i_{2n}, i_{2n+1}) \in (\mathbb{N} \cup \{ 0 \})^{2n+1}$, we set
\[
|I| = i_1 + i_2 + \cdot \cdot \cdot + i_{2n} + i_{2n+1}, \hspace{.5cm} d(I) = i_1 + i_2 + \cdot \cdot \cdot + i_{2n} + 2 \, i_{2n+1}.
\]
The amount $|I|$ is called the length of $I$ and $d(I)$ the homogeneous degree of $I$. We adopt the following multiindex notation for
higher order derivatives and for monomials on $\mathbb{H}^{n}$. If $I=(i_1, i_2, ..., i_{2n+1})$ is a multiindex, 
$X = \{ X_i \}_{i=1}^{2n+1}$, and $z = (x,t) = (x_1, ..., x_{2n}, t) \in \mathbb{H}^{n}$, we put
\[
X^{I} := X_{1}^{i_1} X_{2}^{i_2} \cdot \cdot \cdot X_{2n+1}^{i_{2n+1}}, \,\,\,\,\,\, \text{and} \,\,\,\,\,\,
z^{I} := x_{1}^{i_1} \cdot \cdot \cdot x_{2n}^{i_{2n}} \cdot t^{i_{2n+1}}.
\]
A computation give
\[
X^{I}(f(r \cdot z)) = r^{d(I)} (X^{I}f)(r\cdot z), \,\,\,\,\,\, \text{and} \,\,\,\,\,\, (r\cdot z)^{I} = r^{d(I)} z^{I}.
\]
So, the operators $X^{I}$ and the monomials $z^{I}$ are homogeneous of degree $d(I)$.

Next, we establish an estimate for higher order derivatives (in the variable $(x,t)$) of the kernel 
$\prod_{j=1}^{m} \rho \left( (A_j y, r_j^{-2} s)^{-1} \cdot (x, t) \right)^{-\alpha_j}$.

\begin{proposition} \label{Der Kernel}
For every $N \in \mathbb{N}$, \\

$
\displaystyle{\left| X^{I} \left( \prod_{j=1}^{m} \rho \left( (A_j y, r_j^{-2} s)^{-1} \cdot (x, t) \right)^{-\alpha_j} \right) \right|}
$
\[
\leq C 
\left(\prod_{j=1}^{m} \rho \left( (A_j y, r_j^{-2} s)^{-1} \cdot ( x, t) \right)^{-\alpha_j} \right) 
\left( \sum_{j=1}^{m} \rho \left( (A_j y, r_j^{-2} s)^{-1} \cdot ( x, t) \right)^{-1} \right)^{d(I)}
\]
holds for all multiindex $I$ such that $d(I) \leq N$, where $C$ is a positive constant which does not depend on $(y,s)$ and $(x,t)$.
\end{proposition}

\begin{proof} We fix $N \in \mathbb{N}$. From the left-invariance and the homogeneity of $X^{I}$ and the homogeneity of the function 
$\rho(\cdot)^{-\alpha_j}$ we obtain that
\begin{equation} \label{estim XI}
\left| X^{I} \left(\rho \left( (A_j y, r_j^{-2} s)^{-1} \cdot (x, t) \right)^{-\alpha_j} \right) \right| \leq C_{I, \, j} \, 
\rho \left( (A_j y, r_j^{-2} s)^{-1} \cdot (x, t) \right)^{-\alpha_j - d(I)},
\end{equation}
for every  multiindex $I$ such that $d(I) \leq N$ and every $1 \leq j \leq m$. For $m=2$, (\ref{estim XI}) leads to \\

$
\left|X^{I} \left( \displaystyle{\prod_{j=1}^{2}} \rho \left((A_j y, r_j^{-2} s)^{-1} \cdot (x, t) \right)^{-\alpha_j} \right) \right|
$
\[
\leq C 
\left(\prod_{j=1}^{2} \rho \left( (A_j y,r_j^{-2} s)^{-1} \cdot (x, t) \right)^{-\alpha_j} \right) 
\left( \sum_{j=1}^{2} \rho \left( (A_j y, r_j^{-2} s)^{-1} \cdot (x, t) \right)^{-1} \right)^{d(I)},
\]
for every multiindex $I$ such that $d(I) \leq N$. Finally, applying induction on $m$, the proposition follows.
\end{proof}

\section{Basics on variable Lebesgue and Hardy spaces}

Let $p(\cdot) : \mathbb{H}^{n} \to (0, \infty)$ be a measurable function. Given a measurable set $E \subset \mathbb{H}^{n}$, let
\[
p_{-}(E) = \essinf_{ (x,t) \in E } p(x,t), \,\,\,\, \text{and} \,\,\,\, p_{+}(E) = \esssup_{(x,t) \in E} p(x,t).
\]
When $E = \mathbb{H}^{n}$, we will simply write $p_{-} := p_{-}(\mathbb{H}^{n})$, $p_{+} := p_{+}(\mathbb{H}^{n})$ and 
$\underline{p} := \min \{ p_{-}, 1 \}$. Such function $p(\cdot)$ is called an exponent function.

We say that an exponent function $p(\cdot) : \mathbb{H}^{n} \to (0, \infty)$ such that $0 < p_{-} \leq p_{+} < \infty$ belongs 
to $\mathcal{P}^{\log}(\mathbb{H}^{n})$, if there exist positive constants $C$, $C_{\infty}$ and $p_{\infty}$ such that $p(\cdot)$ satisfies the local log-H\"older continuity condition, i.e.:
\begin{equation} \label{log 1}
|p(x,t) - p(y,s)| \leq \frac{C}{-\log(\rho((x,t)^{-1} \cdot (y,s))}, \,\,\, \text{for} \,\, \rho((x,t)^{-1} \cdot (y,s)) \leq \frac{1}{2},
\end{equation}
and is log-H\"older continuous at infinity, i.e.:
\begin{equation} \label{log 2}
|p(x,t) - p_{\infty}| \leq \frac{C_{\infty}}{\log(e+\rho(x,t))}, \,\,\, \text{for all} \,\, (x,t) \in \mathbb{H}^{n}.
\end{equation}
Here $\rho$ is the {\it Koranyi norm} given by (\ref{Koranyi norm}).

\

Given an exponent function $p(\cdot)$ on $\mathbb{H}^{n}$, we define the variable Lebesgue space $L^{p(\cdot)} = L^{p(\cdot)}(\mathbb{H}^{n})$ to be the set of all measurable functions 
$f: \mathbb{H}^{n} \to \mathbb{C}$ such that for some $\lambda > 0$ 
\[
\int_{\mathbb{H}^{n}} |f(x,t)/\lambda|^{p(x,t)} dx \, dt < \infty.
\]
This becomes a quasi normed space when equipped with the Luxemburg norm
\[
\| f \|_{L^{p(\cdot)}} = \inf \left\{ \lambda > 0 : \int_{\mathbb{H}^{n}} |f(x,t)/\lambda|^{p(x,t)} dx \, dt \leq 1 \right\}.
\]
The following result follows from the definition of the $L^{p(\cdot)}$ - norm.

\begin{lemma} \label{potencia s}
Given a measurable function $p(\cdot) : \mathbb{H}^{n} \to (0, \infty)$ with $0 < p_{-} \leq p_{+} < \infty$, then 
\\
(i) $\| f \|_{L^{p(\cdot)}} \geq 0$ and $\| f \|_{L^{p(\cdot)}} = 0$ if and only if $f \equiv 0$ a.e.,
\\
(ii) $\| c f \|_{L^{p(\cdot)}} = |c| \| f \|_{L^{p(\cdot)}}$ for all $f \in L^{p(\cdot)}$ and all $c \in \mathbb{C}$,
\\
(iii) $\| f +  g \|_{L^{p(\cdot)}} \leq 2^{1/\underline{p} - 1}(\| f \|_{L^{p(\cdot)}} + \| g \|_{L^{p(\cdot)}})$ for all 
$f, g \in L^{p(\cdot)}$,
\\
(iv) $\| f \|_{L^{p(\cdot)}}^{s} = \| |f|^{s} \|_{L^{p(\cdot)/s}}$ for every $s > 0$.
\end{lemma}

Next we shall define the variable Hardy spaces on $\mathbb{H}^{n}$. For them, we introduce the Schwartz space $\mathcal{S}(\mathbb{H}^{n})$ which is defined by
\[
\mathcal{S}(\mathbb{H}^{n}) = \left\{ \phi \in C^{\infty}(\mathbb{H}^{n}) : \sup_{(x,t) \in \mathbb{H}^{n}} (1+\rho(x,t))^{N} 
|(X^{I}f)(x,t)| < \infty \,\,\, \forall \,\, N \in \mathbb{N}_{0}, \, I \in (\mathbb{N}_{0})^{2n+1}  \right\}.
\]
We topologize the space $\mathcal{S}(\mathbb{H}^{n})$ with the following family of seminorms
\[
\| f \|_{\mathcal{S}(\mathbb{H}^{n}), \, N} = \sum_{d(I) \leq N} \sup_{(x,t) \in \mathbb{H}^{n}} (1+\rho(x,t))^{N} |(X^{I}f)(x,t)| \,\,\,\,\,\,\, (N \in \mathbb{N}_{0}),
\]
with $\mathcal{S}'(\mathbb{H}^{n})$ we denote the dual space of $\mathcal{S}(\mathbb{H}^{n})$. 

Given $N \in \mathbb{N}$, define 
\[
\mathcal{F}_{N}=\left\{ \varphi \in \mathcal{S}(\mathbb{H}^{n}):\sum\limits_{d(I) \leq N}
\sup\limits_{(x,t) \in \mathbb{H}^{n}}\left( 1 + \rho(x,t) \right)^{N} | (X^{I} \varphi) (x,t) | \leq 1\right\}.
\]
For any $f \in \mathcal{S}'(\mathbb{H}^{n})$, the grand maximal function of $f$ is given by 
\[
\mathcal{M}_N f(x,t)=\sup\limits_{s>0}\sup\limits_{\varphi \in \mathcal{F}_{N}}\left\vert \left( f \ast \varphi_s \right)(x,t) \right\vert,
\]
where $\varphi_s(x,t) = s^{-2n-2} \varphi(s^{-1} \cdot (x, t))$. 

\begin{definition} \label{Dp def} Given an exponent function $p(\cdot):\mathbb{H}^{n} \to ( 0,\infty)$ with $0 < p_{-} \leq p_{+} < \infty$, we define the integer $\mathcal{D}_{p(\cdot)}$ by
\[
\mathcal{D}_{p(\cdot)} := \min \{ k \in \mathbb{N} \cup \{0\} : (2n+k+3) p_{-} > 2n+2 \}.
\]
For $N \geq \mathcal{D}_{p(\cdot)} + 1$, define 
the variable Hardy space $H^{p(\cdot)}(\mathbb{H}^{n})$ to be the collection of $f \in \mathcal{S}'(\mathbb{H}^{n})$ such that
$\| \mathcal{M}_N f \|_{L^{p(\cdot)}(\mathbb{H}^{n})} < \infty$. Then, the "norm" on the space $H^{p(\cdot)}(\mathbb{H}^{n})$ is taken to be
$\| f \|_{H^{p(\cdot)}} := \| \mathcal{M}_Nf \|_{L^{p(\cdot)}}$.
\end{definition}

\begin{definition} \label{atom def} Let $p(\cdot):\mathbb{H}^{n} \to ( 0,\infty)$, $0 < p_{-} \leq p_{+} < \infty $, and $p_{0} > 1$. 
Fix an integer $D \geq \mathcal{D}_{p(\cdot)}$. A measurable function $a(\cdot)$ on $\mathbb{H}^{n}$ is called a $(p(\cdot), p_{0}, D)$ - atom centered at a $\rho$ - ball $B=B_{\delta}(x_0, t_0)$ if it satisfies the following conditions:

$a_{1})$ $\supp ( a ) \subset B$,

$a_{2})$ $\| a \|_{L^{p_{0}}(\mathbb{H}^{n})} \leq 
\displaystyle{\frac{| B |^{\frac{1}{p_{0}}}}{\| \chi _{B} \|_{L^{p(\cdot)}(\mathbb{H}^{n})}}}$,

$a_{3})$ $\displaystyle{\int_{\mathbb{H}^{n}}} a(z) \, z^{I} \, dz = 0$ for all multiindex $I$ such that $d(I) \leq D$ (here $z=(x,t)$).
\end{definition}

Indeed, every $(p(\cdot), p_{0}, D)$ - atom $a(\cdot)$ belongs to $H^{p(\cdot)}(\mathbb{H}^{n})$. Moreover, there exists an universal constant
$C > 0$ such that $\| a \|_{H^{p(\cdot)}} \leq C$ for all $(p(\cdot), p_{0}, D)$ - atom $a(\cdot)$.

\begin{remark} \label{atomo trasladado}
It is easy to check that if $a(\cdot)$ is a $(p(\cdot), p_{0}, D)$ - atom centered at the ball $B_{\delta}(x_0, t_0)$, then the function 
$a_{(x_0, t_0)}(\cdot) := a((x_0, t_0) \cdot (\cdot))$ is a $(p(\cdot), p_{0}, D)$ - atom centered at the ball $B_{\delta}(0, 0)$.
\end{remark}

\begin{definition} Let $p(\cdot):\mathbb{H}^{n} \to ( 0,\infty)$ be an exponent function such that $0 < p_{-} \leq p_{+} < \infty$.
Given a sequence of nonnegative numbers $\{ \lambda_j \}_{j=1}^{\infty}$ and a family of $\rho$ - balls $\{ B_j \}_{j=1}^{\infty}$, we define
\begin{equation} \label{cantidad A}
\mathcal{A} \left( \{ \lambda_j \}_{j=1}^{\infty}, \{ B_j \}_{j=1}^{\infty}, p(\cdot) \right) := 
\left\| \left\{ \sum_{j=1}^{\infty} \left( \frac{\lambda_j  \chi_{B_j}}{\| \chi_{B_j} \|_{L^{p(\cdot)}}} \right)^{\underline{p}} 
\right\}^{1/\underline{p}} \right\|_{L^{p(\cdot)}}.
\end{equation}
\end{definition}

To prove our main result we will need the following additional estimate.

\begin{lemma} \label{B star}
Let $A \in SO(2n)$, $r >0$, $p(\cdot) : \mathbb{H}^{n} \rightarrow (0, \infty) \in \mathcal{P}^{\log}(\mathbb{H}^{n})$ such that 
$0 < p_{-} \leq p_{+} < \infty$, and let $\{ B_j = B_{\delta_j}(x_j^{0}, t_j^{0}) \}_{j=1}^{\infty}$ be a sequence of $\rho$ - balls 
of $\mathbb{H}^{n}$. 

$(i)$ If $p(A x, t) = p(x,t)$ for all $(x, t) \in \mathbb{H}^{n}$ and $B_j^{\ast} = B_{\gamma \delta_j}(A x_j^{0}, t_j^{0})$ for 
each $j \geq 1$, where $\gamma \geq 1$, then
$$\mathcal{A}\left( \left\{ \lambda_{j}\right\}_{j=1}^{\infty },\left\{ B_{j}^{\ast}\right\} _{j=1}^{\infty }, p(\cdot)\right) \lesssim
\mathcal{A}\left( \left\{ \lambda_{j}\right\}_{j=1}^{\infty },\left\{ B_{j}\right\} _{j=1}^{\infty }, p(\cdot)\right),$$
holds for all sequences of nonnegative numbers $\{ \lambda_j \}_{j=1}^{\infty}$ and $\rho$ - balls $\{ B_j \}_{j=1}^{\infty}$.

$(ii)$ If $p(r x, r^{2}t) = p(x,t)$ for all $(x, t) \in \mathbb{H}^{n}$ and $B_j^{\ast} = B_{\gamma \delta_j}(r x_j^{0}, r^{2} t_j^{0})$ for each $j \geq 1$, where $\gamma \geq 1$, then
$$\mathcal{A}\left( \left\{ \lambda_{j}\right\}_{j=1}^{\infty },\left\{ B_{j}^{\ast}\right\} _{j=1}^{\infty }, p(\cdot)\right) \lesssim
\mathcal{A}\left( \left\{ \lambda_{j}\right\}_{j=1}^{\infty },\left\{ B_{j}\right\} _{j=1}^{\infty }, p(\cdot)\right),$$
holds for all sequences of nonnegative numbers $\{ \lambda_j \}_{j=1}^{\infty}$ and $\rho$ - balls $\{ B_j \}_{j=1}^{\infty}$.

\end{lemma}

\begin{proof}
We start proving $(i)$. Since $p(Ax, t)=p(x, t)$ for all $(x,t) \in \mathbb{H}^{n}$ and $A \in SO(2n)$, a change of variable gives
$$\mathcal{A}\left( \left\{ \lambda_{j}\right\} _{j=1}^{\infty },\left\{
B_{j}^{\ast}\right\} _{j=1}^{\infty }, p(\cdot)\right) =\left\Vert \left\{
\sum\limits_{j=1}^{\infty }\left( \frac{\lambda_{j}\chi _{B_{j}^{\ast}}}{\left\Vert \chi
_{B_{j}\ast}\right\Vert _{p(\cdot)}}\right) ^{\underline{p}}\right\} ^{\frac{1}{%
\underline{p}}}\right\Vert _{p(\cdot)}$$
$$\approx \left\Vert \left\{
\sum\limits_{j=1}^{\infty }\left( \frac{\lambda_{j}\chi _{\gamma B_{j}}}{\left\Vert \chi
_{\gamma B_{j}}\right\Vert _{p(\cdot)}}\right) ^{\underline{p}}\right\} ^{\frac{1}{
\underline{p}}}\right\Vert _{p(\cdot)} = : \mathcal{A}\left( \left\{ \lambda_{j}\right\} _{j=1}^{\infty },\left\{
\gamma B_{j}\right\} _{j=1}^{\infty }, p( \cdot)\right),$$
where $\gamma B_j$ is the $\rho$ - ball with the same center as $B_j$ but whose radius is expanded by the factor $\gamma \geq 1$.
It is clear that $B_j \subset \gamma B_{j}$ for all $j$. Now, the argument utilized in the proof of Lemma 4.8 in \cite{Nakai} works on 
$\mathbb{H}^{n}$ as well, but taking into account \cite[Theorem 4.2]{Fang} instead of \cite[Lemma 2.4]{Nakai}. So
$$\mathcal{A}\left( \left\{ \lambda_{j}\right\} _{j=1}^{\infty },\left\{
\gamma B_{j} \right\} _{j=1}^{\infty }, p(\cdot)\right) \lesssim \mathcal{A}\left( \left\{ \lambda_{j}\right\} _{j=1}^{\infty },\left\{
B_{j}\right\} _{j=1}^{\infty }, p(\cdot)\right).$$
The proof of $(ii)$ is similar. 
\end{proof}

To get our main result we also need the following version of the atomic decomposition for $H^{p(\cdot)}(\mathbb{H}^{n})$ obtained 
in \cite{Fang} (see also \cite[Theorem 5.5]{Pablo}).

\begin{theorem} \label{atomic decomp}
Let $1 < p_0 < \infty$, $p(\cdot) \in \mathcal{P}^{\log}(\mathbb{H}^{n})$ with $0 < p_{-} \leq p_{+} < \infty$. Then, for every $f \in H^{p(\cdot)}(\mathbb{H}^{n}) \cap L^{p_0}(\mathbb{H}^{n})$ and
every integer $D \geq \mathcal{D}_{p(\cdot)}$ fixed, there exist a sequence of nonnegative numbers 
$\{ \lambda_j \}_{j=1}^{\infty}$, a sequence of $\rho$ - balls $\{ B_j \}_{j=1}^{\infty}$ with the bounded intersection property and 
$(p(\cdot), p_0, D)$ - atoms $a_j$ supported on $B_j$ such that $f = \displaystyle{\sum_{j=1}^{\infty} \lambda_j a_j}$ converges 
in $L^{p_0}(\mathbb{H}^{n})$ and
\begin{equation} \label{Hp norm atomic}
\mathcal{A} \left( \{ \lambda_j \}_{j=1}^{\infty}, \{ B_j \}_{j=1}^{\infty}, p(\cdot) \right) 
\lesssim \| f \|_{H^{p(\cdot)}(\mathbb{H}^{n})},
\end{equation}
where the implicit constant in (\ref{Hp norm atomic}) is independent of $\{ \lambda_j \}_{j=1}^{\infty}$, $\{ B_j \}_{j=1}^{\infty}$, and 
$f$.
\end{theorem}

\begin{proposition} \label{dense set}
Let $1 < p_0 < \infty$ and $p(\cdot) \in \mathcal{P}^{\log}(\mathbb{H}^{n})$ with $0 < p_{-} \leq p_{+} < \infty$. Then
$H^{p(\cdot)}(\mathbb{H}^{n}) \cap L^{p_0}(\mathbb{H}^{n}) \subset H^{p(\cdot)}(\mathbb{H}^{n})$ densely.
\end{proposition}

\begin{proof}
The proof is similar to the one given in \cite[see p. 3693]{Nakai}.
\end{proof}

\section{Main result}

In this section we will prove that the generalized Riesz operator $T_{\alpha, m}$ on $\mathbb{H}^{n}$ can be extended to a bounded operator from variable Hardy spaces into variable Lebesgue spaces. The main tools that we will use are Theorem \ref{Hp0-Lq0}, Proposition \ref{Der Kernel} and the atomic decomposition of $H^{p(\cdot)}(\mathbb{H}^{n})$ established in Theorem \ref{atomic decomp}.

\begin{theorem} \label{Hp-Lq}
Let $0 \leq \alpha < Q$, $m \in \mathbb{N} \cap (1- \frac{\alpha}{Q}, \infty)$ and let $T_{\alpha, \, m}$ be the operator 
defined by (\ref{Tm}). Suppose $p(\cdot) \in \mathcal{P}^{\log}(\mathbb{H}^{n})$ 
such that $p(A_j x, r_j^{-2} t) = p(x,t)$ for every $j=1, ..., m$ and $0 < p_{-} \leq p_{+} < \frac{Q}{\alpha}$. 
If $\frac{1}{q(\cdot)} = \frac{1}{p(\cdot)} - \frac{\alpha}{Q}$, then $T_{\alpha, \, m}$ can be extended to a bounded 
operator from $H^{p(\cdot)}(\mathbb{H}^{n})$ into $L^{q(\cdot)}(\mathbb{H}^{n})$.
\end{theorem}

\begin{proof} We take $N\in \mathbb{N}$ such that $0 < \frac{Q}{Q+N} < p_{-}$, thus $N-1 \geq \mathcal{D}_{p(\cdot)}$. Then, 
given $f \in H^{p(\cdot)}(\mathbb{H}^{n}) \cap L^{p_0} (\mathbb{H}^{n})$ (with $p_0 > 1$), by Theorem \ref{atomic decomp} with $D=N-1$, there exist a sequence of nonnegative numbers $\{ \lambda_j \}_{j=1}^{\infty}$, a sequence of 
$\rho$ - balls $\{ B_j \}_{j=1}^{\infty}$ and $(p(\cdot), p_0, N-1)$ atoms $a_j$ supported on $B_j$ such that 
$f = \displaystyle{\sum_{j=1}^{\infty} \lambda_j a_j}$ converges in $L^{p_0}(\mathbb{H}^{n})$ and
\begin{equation} \label{atomic Hp norm}
\mathcal{A} \left( \{ \lambda_j \}_{j=1}^{\infty}, \{ B_j \}_{j=1}^{\infty}, p(\cdot) \right) 
\lesssim \| f \|_{H^{p(\cdot)}(\mathbb{H}^{n})}.
\end{equation}
If $0 < \alpha < Q$, we take $\max\{1, p_{+} \} < p_0 < \frac{Q}{\alpha}$. If $\alpha = 0$, we take $p_0 = 2$. Then, by Theorem 
\ref{Hp0-Lq0}, the operator $T_{\alpha, m}$ is bounded from $L^{p_0}(\mathbb{H}^{n})$ to $L^{q_0}(\mathbb{H}^{n})$ for $1 < p_0 < \frac{Q}{\alpha}$ and $\frac{1}{q_0} = \frac{1}{p_0} - \frac{\alpha}{Q}$. Since $f = \displaystyle{\sum_{j=1}^{\infty} \lambda_j a_j}$ converges in 
$L^{p_0}(\mathbb{H}^{n})$, we have
\[
|T_{\alpha, \, m} f(x,t)| \leq \sum_j \lambda_j |T_{\alpha, \, m} a_j(x,t)|, \,\,\, a.e. \,\, (x,t) \in \mathbb{H}^{n}.
\]
Let $\beta$ be the constant given in \cite[Corollary 1.44]{Folland}, we observe that $\beta \geq 1$ (see \cite[p. 29]{Folland}), and let
$\gamma : = (1 + \max_{1\leq i \leq m} \{ r_i\}) / (\min_{1\leq i \leq m} \{ r_i\})$. Given an atom $a_j$ supported on the $\rho$ - ball $B_j = B_{\delta_j}(x^{0}_j, t^{0}_j)$, for every $i=1, ..., m$ let 
$B_{j \, i}^{\ast} = 2 \beta^{N} \gamma B_{\delta_j}(A_i x^{0}_j, r_{i}^{-2} t^{0}_j )$, where $2\beta^{N} \gamma B$ is the $\rho$ - ball with the same center as $B$ but whose radius is expanded by the factor $2\beta^{N} \gamma$. Now we decompose 
$\mathbb{H}^{n} = \cup_{i=1}^{m} B_{j \, i}^{\ast} \cup \Omega_j$, where $\Omega_j = \left( \cup_{i=1}^{m} B_{j \, i}^{\ast} \right)^{c}$.
Then, 
for $\frac{1}{q(\cdot)} = \frac{1}{p(\cdot)} - \frac{\alpha}{Q}$
\begin{equation} \label{J1 y J2}
\| T_{\alpha, \, m} f \|_{L^{q(\cdot)}} \leq \sum_{i=1}^{m} \left\| \sum_{j=1}^{\infty} \lambda_j \chi_{B_{j \, i}^{\ast}} |T_{\alpha, \, m} a_j| \right\|_{L^{q(\cdot)}}  + \left\| \sum_{j=1}^{\infty} \lambda_j \chi_{\Omega_j} |T_{\alpha, \, m} a_j| \right\|_{L^{q(\cdot)}} =: 
J_1 + J_2, 
\end{equation}
To estimate $J_1$ we once again apply, for the case $0 < \alpha < Q$, Theorem \ref{Hp0-Lq0} with 
$q_0  > \max\{ \frac{Q}{Q - \alpha}, q_{+} \}$ and $\frac{1}{p_0} := \frac{1}{q_0} + \frac{\alpha}{Q}$ (or with $q_0=p_0=2$, if $\alpha =0$). So,
\[
\left\| (T_{\alpha, \, m} a_j)^{q_{*}} \right\|_{L^{q_{0}/q_{*}}(B_{j \, i}^{\ast})} = 
\left\| T_{\alpha,  m} a_j \right\|_{L^{q_{0}}(B_{j \, i}^{\ast})}^{q_{*}} \lesssim \left\| a_j \right\|_{L^{p_{0}}}^{q_{*}} \lesssim 
\frac{| B_j |^{\frac{q_{*}}{p_{0}}}}{\left\| \chi _{B_j }\right\|_{L^{p(\cdot)}}^{q_{*}}} \lesssim 
\frac{\left| B_{j \, i}^{\ast} \right|^{\frac{q_{*}}{q_{0}}}}{\left\| \chi_{B_{j \, i}^{\ast}} \right\|_{L^{q(\cdot)/q_{*}}}},
\]
where $0 < q_{*} < \underline{q}$ is fixed and the last inequality follows from the estimate $\| \chi_B \|_{L^{q(\cdot)}} \approx |B|^{-\alpha/Q} \| \chi_B \|_{L^{p(\cdot)}}$ (which is a consequence of \cite[Lemma 4.1]{Fang}), Lemma \ref{potencia s}-(iv), and Lemma 3.2 in \cite{Pablo} applied to the exponent $q(\cdot)/q_{*}$. Now, since $0 < q_{*} < 1$, we apply the $q_{\ast}$-inequality and Proposition 3.3 in \cite{Pablo} with $b_j = \left( \chi_{B_{j \, i}^{\ast}} \cdot |T_{\alpha, \, m} a_j| \right)^{q_{*}}$, 
$A_j = \left\| \chi_{B_{j \, i}^{\ast}} \right\|_{L^{q(\cdot)/q_{*}}}^{-1}$ and $s= q_0/q_{*}$, to obtain
\[
J_1 \lesssim \sum_{i=1}^{m} \left\| \sum_{j} \left(\lambda_j \, \chi_{B_{j \, i}^{\ast}} \, |T_{\alpha, \, m} a_j| \right)^{q_{*}} \right\|^{1/q_{*}}_{L^{q(\cdot)/q_{*}}} \lesssim \sum_{i=1}^{m} \left\| \sum_{j} \left( \frac{\lambda_j}{\left\| \chi_{B_{j \, i}^{\ast}} \right\|_{L^{q(\cdot)}}} \right)^{q_{*}} \chi_{B_{j \, i}^{\ast}} \right\|^{1/q_{*}}_{L^{q(\cdot)/q_{*}}}.
\]
From this inequality, Lemma 5.7 in \cite{Pablo} and Lemma \ref{B star} applied to $q(\cdot)$ give
\[
J_1 \lesssim \sum_{i=1}^{m} \mathcal{A}\left( \{ \lambda_j \}_{j=1}^{\infty}, \{ B_{j \, i}^{\ast} \}_{j=1}^{\infty}, q(\cdot) \right) 
\lesssim
m \, \mathcal{A}\left( \{ \lambda_j \}_{j=1}^{\infty}, \{ B_{j} \}_{j=1}^{\infty}, q(\cdot) \right).
\]
Then, Proposition 5.8 in \cite{Pablo} and (\ref{atomic Hp norm}) give
\begin{equation} \label{J1}
J_1 \lesssim \mathcal{A}\left( \{ \lambda_j \}_{j=1}^{\infty}, \{ B_j \}_{j=1}^{\infty}, q(\cdot) \right)
\lesssim \mathcal{A}\left( \{ \lambda_j \}_{j=1}^{\infty}, \{ B_j \}_{j=1}^{\infty}, p(\cdot) \right) 
\lesssim \| f \|_{H^{p(\cdot)}}.
\end{equation}
Now, we proceed to estimate $J_2$. For them, we first consider a single $(p(\cdot), p_0, N-1)$ - atom $a(\cdot)$ supported on the $\rho$ - ball 
$B = B_{\delta}(x_0, t_0)$ and let $\widetilde{\Omega} = \left( \cup_{i=1}^{m} B_i^{\ast} \right)^{c}$, where 
$B_{i}^{\ast} = 2 \beta^{N} \gamma B_{\delta}(A_i x_0, r_i^{-2}t_0)$. Then, we decompose 
$\widetilde{\Omega} = \cup_{k=1}^{m} \widetilde{\Omega}_k$ with
\[
\widetilde{\Omega}_k = \left\{(x,t) \in \widetilde{\Omega} : \rho \left( (A_k^{-1} x, r_k^{2} t)^{-1} \cdot (x_0, t_0) \right) \leq 
\rho \left( (A_j^{-1} x, r_j^{2} t)^{-1} \cdot (x_0, t_0) \right) \,\, \text{for all} \,\, j \neq k \right\}.
\]
Remark \ref{cambio de centro} and a change of variable lead to
\begin{eqnarray*}
T_{\alpha, m} a(x,t) &=& \int_{B_{\delta}(x_0, t_0)} a(y,s) \prod_{j=1}^{m} \rho \left((A_jy, r_j^{-2} s)^{-1} \cdot (x,t) \right)^{-\alpha_j} 
\, dy \, ds \\
&=& C \int_{B_{\delta}(0, 0)} a((x_0, t_0) \cdot (y,s)) \prod_{j=1}^{m} \rho \left((A_j^{-1} x, r_j^{2} t)^{-1} \cdot (x_0, t_0) \cdot (y,s) \right)^{-\alpha_j} \, dy \, ds, 
\end{eqnarray*}
for all $(x, t) \in  \widetilde{\Omega}$, where $C= r_1^{-\alpha_1} \cdot \cdot \cdot r_m^{-\alpha_1}$. By Remark \ref{atomo trasladado}, for each $(x, t) \in  \widetilde{\Omega}$, it follows that
\begin{equation} \label{Taz}
T_{\alpha, m} a(x,t) = C \int_{B_{\delta}(0, 0)} a((x_0, t_0) \cdot (y,s)) 
\end{equation}
\[
\times \left[\prod_{j=1}^{m} \rho \left((A_j^{-1} x, r_j^{2} t)^{-1} \cdot (x_0, t_0) \cdot (y,s) \right)^{-\alpha_j} - q_{N}(y,s) \right] \, dy \, ds, 
\]
where $(y,s) \to q_N(y,s)$ is the left Taylor polynomial of the function 
\[
(y,s) \to \prod_{j=1}^{m} \rho \left((A_j^{-1} x, r_j^{2} t)^{-1} \cdot (x_0, t_0) \cdot (y,s) \right)^{-\alpha_j} 
\]
at $e=(0,0)$ of homogeneous degree $N-1$. Then by the left Taylor inequality in \cite[Corollary 1.44]{Folland},
\begin{equation} \label{Taylor ineq}
\left| \prod_{j=1}^{m} \rho \left((A_j^{-1} x, r_j^{2} t)^{-1} \cdot (x_0, t_0) \cdot (y,s) \right)^{-\alpha_j} - q_{N}(y,s) \right| 
\end{equation}
\[
\lesssim \rho(y,s)^{N} 
\sup_{\rho(z,u) \leq \beta^{N}\rho(y,s), \, d(I)=N} \left|X^{I} \left(
\prod_{j=1}^{m} \rho \left((A_j^{-1} x, r_j^{2} t)^{-1} \cdot (x_0, t_0) \cdot (z,u) \right)^{-\alpha_j} \right) \right|.
\]
Now, for $(y,s) \in B_{\delta}(0, 0)$, $(x, t)^{-1} \cdot (A_k x_0, r_k^{-2} t_0) \notin  2\beta^{N} \gamma B_{\delta}(0, 0)$ and 
$\rho(z,u) \leq \beta^{N} \rho(y,s)$, we have $\rho((A_k^{-1} x, r_k^{2} t)^{-1} \cdot (x_0, t_0)) \geq 2 r_k \gamma \rho(z,u)$, 
where $r_k \gamma > 1$, and hence 
\begin{equation} \label{omega k}
\rho \left( (A_k^{-1} x, r_k^{2} t)^{-1} \cdot (x_0, t_0) \cdot (z,u) \right) \geq \frac{2r_k \gamma -1}{2r_k \gamma}
\rho((A_k^{-1} x, r_k^{2} t)^{-1} \cdot (x_0, t_0)).
\end{equation}
Then for $(x,t) \in \widetilde{\Omega}_k$, by (\ref{Taylor ineq}), Proposition \ref{Der Kernel} applied on the variable $(z,u)$ and (\ref{omega k}), we get
\[
\left| \prod_{j=1}^{m} \rho \left((A_j^{-1} x, r_j^{2} t)^{-1} \cdot (x_0, t_0) \cdot (y,s) \right)^{-\alpha_j} - q_{N}(y,s) \right| 
\lesssim \delta^{N} \rho \left( (A_k^{-1} x, r_k^{2} t)^{-1} \cdot ( x_0, t_0) \right)^{\alpha - Q - N}.
\]
This estimate and (\ref{Taz}) lead to
\begin{eqnarray*}
|T_{\alpha, m}a(x,t)| &\lesssim& \delta^{N} \rho \left( (A_k^{-1} x, r_k^{2} t)^{-1} \cdot ( x_0, t_0) \right)^{\alpha-Q-N} \| a \|_{L^{1}} \\
&\lesssim& \delta^{N} \rho\left( (A_k^{-1} x, r_k^{2} t)^{-1} \cdot ( x_0, t_0) \right)^{\alpha-Q-N} |B|^{1-\frac{1}{p_0}} \| a \|_{L^{p_0}} \\
&\lesssim& \frac{\delta^{N+Q}}{\| \chi_{B} \|_{L^{p(\cdot)}}} \rho\left( (A_k^{-1} x, r_k^{2} t)^{-1} \cdot ( x_0, t_0) \right)^{\alpha-Q-N} \\
&\lesssim& \frac{\left( M_{\frac{\alpha Q}{Q+N}}(\chi_{B})(A_k^{-1} x, r_k^{2} t) \right)^{\frac{Q+N}{Q}}}{\| \chi_{B} \|_{L^{p(\cdot)}}}, \,\,\,\, \forall \,\, (x,t) \in \widetilde{\Omega}_k.
\end{eqnarray*}
Finally, by applying the above reasoning with $a_j(\cdot)$ instead of $a(\cdot)$, $\widetilde{\Omega} = \Omega_j$ and $j \in \mathbb{N}$, we have that $\Omega_j = \cup_{k=1}^{m} \Omega_{j \, k}$ and
\begin{equation} \label{L2 estimate}
|T_{\alpha, m}a_j(x,t)| \lesssim \frac{\left( M_{\frac{\alpha Q}{Q+N}}(\chi_{B_j})(A_k^{-1} x, r_{k}^{2}t) \right)^{\frac{Q+N}{Q}}}{\| \chi_{B_j} \|_{L^{p(\cdot)}}}, \,\,\,\, \text{for all} \,\, (x,t) \in \Omega_{j \, k}.
\end{equation}
From (\ref{L2 estimate}), (\ref{J1 y J2}) and that $q(A_k x, r_{k}^{-1} t) = q(x,t)$ for all $(x,t)$ and all $k=1, ..., m$, we obtain
\[
J_2 \lesssim \sum_{k=1}^{m} \left\| \sum_j \lambda_j \chi_{\Omega_{j \, k}} |T_{\alpha, m}a_j(\cdot)| \right\|_{L^{q(\cdot)}}
\lesssim \sum_{k=1}^{m} \left\| \sum_j \lambda_j \chi_{\Omega_{j \, k}}
\frac{\left( M_{\frac{\alpha Q}{Q+N}}(\chi_{B_j})(A_k^{-1}\cdot, r_k^{2} \cdot) \right)^{\frac{Q+N}{Q}}}{\| \chi_{B_j} \|_{L^{p(\cdot)}}} \right\|_{L^{q(\cdot)}}
\]
\[ 
\lesssim m \left\| \left\{ \sum_j \lambda_j 
\frac{\left( M_{\frac{\alpha Q}{Q+N}}(\chi_{B_j})(\cdot) \right)^{\frac{Q+N}{Q}}}{\| \chi_{B_j} \|_{L^{p(\cdot)}}} \right\}^{\frac{Q}{Q+N}} \right\|_{L^{\frac{Q+N}{Q}q(\cdot)}}^{\frac{Q+N}{Q}}.
\]
Since $1 < \frac{Q+N}{Q} p_{-} \leq \frac{Q+N}{Q} p_{+} < \frac{Q+N}{\alpha}$, Theorem 4.4 in \cite{Pablo} gives
\begin{equation} \label{J2}
J_2 \lesssim 
\left\| \left\{ \sum_j \lambda_j 
\frac{\chi_{B_j}}{\| \chi_{B_j} \|_{L^{p(\cdot)}}} \right\}^{\frac{Q}{Q+N}} \right\|_{L^{\frac{Q+N}{Q}p(\cdot)}}^{\frac{Q+N}{Q}} \lesssim
\left\| \sum_j \lambda_j \frac{\chi_{B_j}}{\| \chi_{B_j} \|_{L^{p(\cdot)}}} \right\|_{L^{p(\cdot)}}
\end{equation}
\[
\lesssim
\left\| \left\{ \sum_j \left( \lambda_j \frac{\chi_{B_j}}{\| \chi_{B_j} \|_{L^{p(\cdot)}}} \right)^{\underline{p}}\right\}^{1/\underline{p}} 
\right\|_{L^{p(\cdot)}} = \mathcal{A}\left( \{ \lambda_j \}_{j=1}^{\infty}, \{ B_j \}_{j=1}^{\infty}, p(\cdot) \right) 
\lesssim \| f \|_{H^{p(\cdot)}}.
\]
Then, (\ref{J1 y J2}), (\ref{J1}) and (\ref{J2}) allow us to conclude that
\[
\| T_{\alpha, \, m} f \|_{L^{q(\cdot)}(\mathbb{H}^{n})} \lesssim \| f \|_{H^{p(\cdot)}(\mathbb{H}^{n})},
\]
for all $f \in H^{p(\cdot)}(\mathbb{H}^{n}) \cap L^{p_0}(\mathbb{H}^{n})$. So the theorem follows from Proposition \ref{dense set}.
\end{proof}

\begin{example}
Given $0 < \alpha < Q$, let $r : \mathbb{R} \to (0, \, \infty)$ be an exponent function such that $0 < r_{-} \leq r_{+} < \frac{Q}{\alpha}$ and $r(\cdot)$ is log-H\"older continuous (locally and at infinite) on $\mathbb{R}$, i.e.: 
\begin{equation} \label{log 3}
|r(s) - r(u)| \leq \frac{C}{-\log(|s-u|)} \,\,\,\, \text{for} \,\, |s-u|\leq \frac{1}{2},
\end{equation}
and
\begin{equation} \label{log 4}
|r(s) - r_{\infty}| \leq \frac{C_{\infty}}{\log(e + |s|)} \,\,\,\, \text{for all} \,\, s \in \mathbb{R}.
\end{equation}
Let $p : \mathbb{H}^{n} \to (0, \, \infty)$ be the exponent function defined by 
\begin{equation} \label{exp func}
p(x,t) = r(\rho(x,t)), \,\,\,\, \text{for all} \,\, (x,t) \in \mathbb{H}^{n},
\end{equation}
where $\rho$ is the Koranyi norm on $\mathbb{H}^{n}$. It is clear that $p(Ax,t) = p(x,t)$ for all $(x,t) \in \mathbb{H}^{n}$ and all
$A \in Sp(2n, \mathbb{R}) \cap SO(2n)$. Since $\rho((x,t)^{-1})= \rho(x,t)$ and $|\rho(x,t) - \rho(y,s)| \leq \rho((x,t) \cdot (y,s))$, (\ref{log 3}) implies (\ref{log 1}). That (\ref{log 4}) implies (\ref{log 2}) it is obvious. So, the exponent function $p(\cdot)$ is log-H\"older continuous (locally and at infinite) on $\mathbb{H}^{n}$. Then, the exponent $p(\cdot)$ defined by (\ref{exp func}) satisfies the hypotheses of Theorem \ref{Hp-Lq}. 
\end{example}


\end{document}